\documentclass[12pt, a4paper]{article}

\usepackage{amsmath}
\usepackage{amsthm}
\usepackage{amssymb}
\usepackage{color}

\newtheorem{theorem}{Theorem}
\newtheorem{proposition}[theorem]{Proposition}
\newtheorem{lemma}[theorem]{Lemma}

\newtheorem{assumption}[theorem]{Assumption}

\title{Small ball probabilities and large deviations for grey Brownian motion}
\author{Stefan Gerhold\\
TU Wien \\
\tt{sgerhold@fam.tuwien.ac.at}
}

\date{\today}

\numberwithin{equation}{section}
\numberwithin{theorem}{section}

\begin{document}

\maketitle

\begin{abstract}
  We show that the uniform norm of generalized grey Brownian motion over the unit
  interval has an analytic density, excluding the special case of fractional Brownian motion.
  Our main result is an asymptotic expansion for the small ball probability
  of generalized grey Brownian motion, which extends
  to other norms on path space. The decay rate is not exponential 
  but polynomial, of degree two. For the uniform norm and the H\"older norm, we also prove a large deviations
  estimate.
\end{abstract}



\section{Introduction}

Generalized grey Brownian motion (ggBm) is a two-parameter stochastic process $B_{\alpha,\beta}$,
which is in general not Gaussian.
Introduced in~\cite{MuMa09,MuPa08}, ggBm has been considered in the physics literature to model
anomalous diffusions with non-Gaussian marginals, including both slow (variance grows
slower than linearly) and fast diffusive behavior.
The process $B_{\alpha,\beta}$ has stationary increments and is self-similar with parameter $H=\alpha/2$ \cite[Proposition~3.2]{MuMa09}.
 The marginal density of ggBm satisfies a fractional partial integro-differential equation~\cite{MuMa09}.
Special cases of ggBm include
fractional Brownian motion (fBm; $\beta= 1$), grey Brownian motion (\cite{Sc90}; $\alpha = \beta$),
 and Brownian motion ($\alpha = \beta = 1$). Our focus is mainly on the case $\beta<1$.
 In~\cite{MuMa09}, a generalized grey noise space is defined, motivated by white noise
 space, but with the Gaussian characteristic function replaced by the Mittag-Leffler function.
 The ggBm is then defined by evaluating generalized grey noise at the test function $1_{[0,t)}$.
 We do not go into details, because for our purposes, the representation
\begin{equation}\label{eq:repr}
  B_{\alpha,\beta}(t) = \sqrt{L_\beta} B_{\alpha/2}(t),\quad 0<\alpha<2,\ 0<\beta< 1,
\end{equation}
which was proved in~\cite{MuPa08}, is more convenient. Here,
$B_{\alpha/2}$ is a fBm with Hurst parameter $H=\alpha/2$,
and~$L_\beta$ is
an independent positive random variable whose density is the $M$-Wright function (see below).
The representation~\eqref{eq:repr} makes sense also in the limiting case $\beta=1$,
but we will not require this.

The problem of small ball probabilities, also called small deviations, consists
of estimating
\begin{equation}\label{eq:sb}
  \mathbb{P}\Big[\sup_{0\leq t\leq 1}|B_{\alpha,\beta}(t)|\leq \varepsilon\Big],\quad \varepsilon \downarrow 0,
\end{equation}
asymptotically.
More generally, we can consider
\[
  \mathbb{P}\big[\| B_{\alpha,\beta} \|\leq \varepsilon\big],\quad \varepsilon \downarrow 0,
\]
where $\|\cdot\|$ is a norm on $C_0^{\gamma}[0,1]$, the space of $\gamma$-H\"older continuous
functions, with $0<\gamma<H=\alpha/2$.
For ggBm with $\beta<1$, our main result (Theorem~\ref{thm:main}) shows 
that~\eqref{eq:sb} is of order $\varepsilon^2$, and that this also
holds for some other norms.
For Gaussian processes, such as fBm ($\beta=1$), the small ball problem has been
studied extensively~\cite{LiSh01}, and exponential decay is typical.
 But there are also many works
studying small ball probabilities for non-Gaussian processes; see, e.g., \cite{AuLiLi09,AuSi07} and the references therein.
We refer to~\cite{Ko16,Na09} for other
examples of processes with the small ball rate~$\varepsilon^2$ of ggBm.

In Section~\ref{se:small}, we will show that the known exponential small ball estimates for fBm
can be used to deduce our quadratic small ball estimate for ggBm. As a byproduct, we show
that the uniform norm (sup norm) of ggBm has a smooth, even analytic, pdf. In Section~\ref{se:large},
we provide a large deviations estimate. The decay rate is exponential, but slower than
Gaussian, depending on the parameter~$\beta$.

Notation: When we write $B_{\alpha,\beta}$, we always
mean the process on the time interval $[0,1]$, i.e.\ $B_{\alpha,\beta}=(B_{\alpha,\beta}(t))_{0\leq t\leq 1}$.
We write $F_H$ for the cdf of $\| B_H\|$, assuming
that the choice of the norm $\| \cdot \|$ is clear from the context. 
As usual, $\mathbb{R}^+=(0,\infty)$ denotes the positive reals.
The letter $C$ denotes various positive constants.

\section{Analyticity of the cdf and small ball probability}\label{se:small}

The $M$-Wright function, which is the pdf of $L_\beta$ in~\eqref{eq:repr}, is defined by
\begin{equation}\label{eq:def M}
   M_\beta(x)= \sum_{n=0}^\infty \frac{(-x)^n}{n! \Gamma(1-\beta-\beta n)},
   \quad x\geq0,\ 0<\beta<1.
\end{equation}
It is not obvious that $M_\beta$ is a pdf; for this, and more information
on $M_\beta$ and its generalizations, we refer to~\cite{MaMuPa10}. For later use,
we note that it follows from Euler's reflection formula that
\begin{equation}\label{eq:refl}
  \frac{1}{\Gamma(1-\beta-\beta n)} = \frac{\sin\big(\pi(\beta+\beta n)\big)}{\pi}
    \Gamma(\beta+\beta n),
\end{equation}
(cf.~\cite[p.~41]{Wr40} and~\cite[(3.8)]{MaMuPa10}), 
which shows, by Stirling's formula for the gamma function, that the series in~\eqref{eq:def M}
defines an entire function.
For this, the crude version
\begin{equation}\label{eq:stir}
  \Gamma(x) = x^{x+o(x)},\quad x\uparrow \infty,
\end{equation}
of Stirling's formula suffices.
 We will also need the asymptotic behavior of~$M_\beta$
  at infinity~\cite[(4.5)]{MaMuPa10},
  \begin{equation}\label{eq:M as}
     M_\beta(x) = \exp\Big( {-\frac{1-\beta}{\beta}}(\beta x)^{\frac{1}{1-\beta}}
     + O(\log x) \Big),\quad x\uparrow \infty.
  \end{equation}

Our main assumption is that fBm satisfies an exponential small ball estimate
w.r.t.\ to the chosen norm $\|\cdot\|$.

\begin{assumption}\label{ass}
  For $0<H<1$, there are $\theta,C_1,C_2>0$ such that
  \begin{equation*}
  -C_1 \varepsilon^{-\theta} \leq  \log\mathbb{P}\big[\| B_H \|\leq \varepsilon\big]\leq
   -C_2 \varepsilon^{-\theta}, \quad \varepsilon \in (0,1].
\end{equation*}
\end{assumption}
For the uniform norm, it is known that this holds with $\theta=1/H$,
\begin{equation}\label{eq:fbm}
  -C_1 \varepsilon^{-1/H} \leq  \log\mathbb{P}\Big[\sup_{0\leq t\leq 1}|B_H(t)|\leq \varepsilon\Big]\leq
   -C_2 \varepsilon^{-1/H}.
\end{equation}
Assumption~\ref{ass} also holds for the $\gamma$-H\"older norm, where
$0<\gamma<H$, and for the $L^2$-norm.
See~\cite{Br03,LiSh01,Li12} for the corresponding values of~$\theta$, and
for much more information on small ball probabilities
for fBm and other Gaussian processes.

The examples we just mentioned are norms in the classical sense, and so we stick to this
terminology in our statements. From our proofs, it is clear that it would suffice throughout
to assume that $\|\cdot\|$ is a measurable non-negative homogeneous functional.
\begin{proposition}\label{prop:an}
  Let $0<\alpha<2$, $0<\beta<1$. If the norm $\| \cdot \|$ satisfies Assumption~\ref{ass},
  then the cdf of $\| B_{\alpha,\beta} \|$ is an analytic function on~$\mathbb{R}^+$. In particular, this holds for the cdf of $\|B_{\alpha,\beta}\|_\infty=\sup_{0\leq t\leq 1}|B_{\alpha,\beta}(t)|$.
\end{proposition}
\begin{proof}
   Recall that $F_H$ denotes the cdf of $\| B_H \|$.
  {}From~\eqref{eq:repr} we find
  \begin{align}
     \mathbb{P}\big[\| B_{\alpha,\beta} \|\leq \varepsilon\big]&=
   \int_0^\infty F_H(\varepsilon x^{-1/2})M_{\beta}(x)dx \notag \\
   &=
      2\varepsilon^2\int_0^\infty F_H(y)M_{\beta}(\varepsilon^2/y^2)
      y^{-3}dy. \label{eq:for hol}
  \end{align}
  As $M_{\beta}$  extends to an entire function (see above),  the     
  last integrand clearly is  an entire function of~$\varepsilon$  for any fixed $y>0$.
  The function $M_{\beta}$ is bounded on~$\mathbb{R}^+$, as follows,
  e.g., from~\eqref{eq:def M} and~\eqref{eq:M as}. Thus, the integrand in~\eqref{eq:for hol}
   can be bounded by an integrable function of~$y$, independently
   of~$\varepsilon$. Thus, the conditions of a standard criterion for
   complex differentiation under the integral 
   sign~\cite[Theorem IV.5.8]{El05} are satisfied, which yields the assertion.
\end{proof}
Note that fBm, i.e.\ $\beta=1$, is not covered by Proposition~\ref{prop:an}.
In~\cite{LaNu03}, it is shown by Malliavin calculus that $\sup_{0\leq t\leq 1}B_H$
(without the absolute value) has a $C^\infty$ density.

We now show that, for $\beta<1$, the small ball ball probability 
of ggBm is of order~$\varepsilon^2$ as $\varepsilon\downarrow0$.
For $2/\theta+\beta<1$ ($\alpha+\beta<1$ for the uniform norm), we express it as a power series,
which yields a full asymptotic expansion.
We write
\[
    \eta_k(H):=\mathbb{E}\big[\| B_H \|^{-k}\big], \quad k\in\mathbb N,
\]
for the negative moments of the norm of fBm, omitting the dependence
on the norm $\|\cdot\|$ in the notation $\eta_k(H)$. By integration by parts,
it is easy to see that $\eta_k(H)$ is finite under Assumption~\ref{ass}.
\begin{theorem}\label{thm:main}
  Let $0<\alpha<2$, $0<\beta<1$, and define $H=\alpha/2$. Under Assumption~\ref{ass}, the small ball probability of ggBm satisfies
    \begin{equation}\label{eq:as e2}
     \mathbb{P}\big[\| B_{\alpha,\beta} \| \leq \varepsilon\big]
        \sim \frac{\eta_2(H)\varepsilon^2}{\Gamma(1-\beta)}, \quad \varepsilon\downarrow0.
   \end{equation}
   If, additionally, $2/\theta+\beta<1$, then it has the convergent series representation
    \begin{equation}\label{eq:expans}
    \mathbb{P}\big[\| B_{\alpha,\beta} \|\leq \varepsilon\big]
   =
    2\sum_{n=0}^\infty \frac{(-1)^n \eta_{2n+2}(H)}{(2n+2)n! \Gamma(1-\beta-\beta n)} \varepsilon^{2n+2},
    \quad \varepsilon\geq 0.
  \end{equation}
  In particular, if $\|\cdot \|=\|\cdot \|_{\infty}$, then~\eqref{eq:expans}
  holds for  $\alpha+\beta<1$.
\end{theorem}
\begin{proof}
    By integration by parts, we have
    \begin{equation}\label{eq:parts}
       \int_0^\infty\frac{F_H(y)}{y^{2n+3}}dy = \frac{1}{2n+2} \int_0^\infty 
       y^{-2n-2}F_H(dy)
  =\frac{\eta_{2n+2}(H)}{2n+2}.
   \end{equation} 
   The assertion~\eqref{eq:as e2} follows from~\eqref{eq:for hol}, \eqref{eq:def M} for $x=0$, \eqref{eq:parts}
   for $n=0$, and dominated convergence, because~$M_\beta$ is a bounded function.
   For the next statement, define
   \[
     G_N(\varepsilon,y):=\sum_{n=N+1}^\infty   \frac{(-1)^n\varepsilon^{2n-2N+1}}{y^{2n}n! \Gamma(1-\beta-\beta n)},\quad
     y>0,\ \varepsilon\in[0,1],
   \]
   so that~\eqref{eq:for hol} yields, for $N\in \mathbb N$,
   \begin{multline}\label{eq:for expans}
     \mathbb{P}\big[\| B_{\alpha,\beta} \| \leq \varepsilon\big]\\
      =
     2\varepsilon^2\int_0^\infty \frac{F_H(y)}{y^3}
          \sum_{n=0}^N  \frac{(-\varepsilon^2/y^2)^n}{n! \Gamma(1-\beta-\beta n)} dy
       +  2\varepsilon^{2N+1}  \int_0^\infty \frac{F_H(y)}{y^3} G_N(\varepsilon,y)dy.
   \end{multline}
   For the finite sum, we can use~\eqref{eq:parts} to rewrite the summands
   as in~\eqref{eq:expans}.  
   We now provide an integrable bound
   for the last integrand in~\eqref{eq:for expans} that does not depend on $\varepsilon\in[0,1]$.
   It is clear that
   \begin{equation}\label{eq:est G}
     |G_N(\varepsilon,y)| \leq \sum_{n=N+1}^\infty   \frac{1}{y^{2n}n! \Gamma(1-\beta-\beta n)},\quad
     y>0,\ \varepsilon\in[0,1].
   \end{equation}
   By~\eqref{eq:refl} and Stirling's formula,
   \begin{equation}\label{eq:from stir}
    \frac{1}{n! |\Gamma(1-\beta -\beta n)|} \leq n^{-(1-\beta)n+o(n)}
    \leq C n^{-(1-\hat{\beta})n}, \quad n\in\mathbb N,
   \end{equation}
   for any $\hat{\beta}>\beta$; we will fix~$\hat \beta$ later.
   {}From~\eqref{eq:est G}, \eqref{eq:from stir}, and Stirling's formula, we conclude
   \begin{align}
      |G_N(\varepsilon,y)| &\leq C \sum_{n=N+1}^\infty \frac{1}{y^{2n}\Gamma((1-\hat{\beta})n)} \notag \\
      &=y^{-2} E_{1-\hat{\beta},1-\hat{\beta}}(y^{-2})
         - \sum_{n=1}^N  \frac{1}{y^{2n}\Gamma((1-\hat{\beta})n)},\label{eq:G bd}
   \end{align}
   where
   \[
     E_{u,v}(z)=\sum_{n=0}^\infty \frac{z^n}{\Gamma(u n+v)} ,\quad u,v>0,\ z\in\mathbb C,
   \]
   denotes the two-parameter Mittag-Leffler function. 
   We now use the uniform bound~\eqref{eq:G bd} in~\eqref{eq:for expans}.
   Integrability at $\infty$ is obvious, and we now show integrability at zero.
   By~\cite[Theorem~4.3]{GoKiMaRo14},
   \[
      E_{1-\hat{\beta},1-\hat{\beta}}(y^{-2})=\exp\Big(y^{-\frac{2}{1-\hat \beta}}\big(1+o(1)\big)\Big),
      \quad y\downarrow0.
   \]
   We see, using Assumption~\ref{ass} for $F_H$, that the last integrand in~\eqref{eq:for expans} satisfies
   \[
       \frac{F_H(y)}{y^3} G_N(\varepsilon,y) \leq 
       \exp\Big({-C_1} y^{-\theta} + y^{-\frac{2}{1-\hat \beta}}
       +o\big(y^{-\big(\theta \wedge \frac{2}{1-\hat \beta}\big)}\big)\Big),
       \quad y\downarrow0,
   \]
   uniformly w.r.t.\ $\varepsilon\in[0,1]$. This is integrable if
   $\theta>2/(1-\hat \beta)$, i.e., $2/\theta+\hat \beta <1$. Clearly,
   our assumption that $2/\theta+ \beta <1$ allows us to chose
   such a $\hat \beta > \beta$.
   By the following lemma, $\eta_{2n+2}(H)=2^{2n/\theta+o(n)}$. Using~\eqref{eq:from stir},
   we can thus take the limit $N\uparrow \infty$ in~\eqref{eq:for expans}
   for fixed $\varepsilon\in[0,1]$,
   which proves~\eqref{eq:expans} for these~$\varepsilon$. The extension to any~$\varepsilon\geq0$
   follows by analytic continuation, using Proposition~\ref{prop:an}.   
\end{proof}

In the preceding proof, we applied the following estimate for negative moments of the supremum
of fBm. Note that moments with \emph{positive} exponent are estimated in~\cite{Pr14}; see also~\cite{NoVa99}.
\begin{lemma}\label{le:mom}
  Under Assumption~\ref{ass},
  for $k\uparrow \infty$, we have $\eta_k(H)=k^{k/\theta+o(k)}$.
\end{lemma}
\begin{proof}
  We show only the upper estimate, as the lower one can be proven analogously.
  By~\eqref{eq:fbm}, there is $\varepsilon_0>0$ such that
  \[
     F_H(y)\leq 2 \exp({-C_2}y^{-\theta}),\quad
     0<y\leq \varepsilon_0.
  \]
  Define $\tilde K := 2 \vee \exp(C_2 \varepsilon_0^{-\theta})$.
  Then,
  \[
    F_H(y) \leq \tilde K \exp({-C_2}y^{-\theta}), \quad y>0;
  \]
  note that the right hand side is $\geq1$ for $y\geq \varepsilon_0$.
  This implies
  \begin{align*}
    \eta_k(H) &= \int_0^\infty y^{-k}F_H(dy) 
    = k \int_0^\infty y^{-k-1}F_H(y)dy \\
    &\leq k \tilde{K} \int_0^\infty \exp({-C_2}y^{-\theta})  y^{-k-1}dy\\
    &= e^{O(k)} \int_0^\infty e^{-w}
    w^{k/\theta -1} dw\\
    &=  e^{O(k)} \Gamma(k/\theta-1)=k^{k/\theta+o(k)},
  \end{align*}
  by Stirling's formula~\eqref{eq:stir} for the gamma function.
\end{proof}

If $2/\theta+ \beta <1$, then the series in~\eqref{eq:expans}
diverges for any $\varepsilon>0$. Indeed,
there is an increasing sequence $(n_j)$
   in~$\mathbb N$ such that the lower bound
   \[
     \mathrm{dist}(1-\beta -\beta n_j,\mathbb Z)
     \geq C >0, \quad j\in\mathbb N,
   \]
   holds.
   For rational~$\beta\in(0,1)$, this is clear by periodicity.
   For irrational~$\beta$, it follows from the classical fact that the sequence
   of fractional parts $\{n\beta\}$ is dense in $[0,1]$
   (Kronecker's approximation theorem).
   Hence, again by~\eqref{eq:refl} and Stirling's formula,
   \[
    \frac{1}{ |\Gamma(1-\beta -\beta n_j)|} \geq n_j^{\beta n_j + o(n_j)},
   \]
   which, together with Lemma~\ref{le:mom}, shows divergence. We leave it as an open problem
   if~\eqref{eq:expans} still holds in the sense of an asymptotic
   expansion of the small ball probability, if  $2/\theta+ \beta \leq1$.

\section{Large deviations}\label{se:large}

For fractional Brownian motion, it is well known that
\begin{equation}\label{eq:fbm ld sup}
  \mathbb{P}\Big[ \sup_{0\leq t\leq 1}|B_H(t)| \geq y\Big]
  =\exp\big({-\tfrac12}y^2+o(y^2)\big),\quad y\uparrow \infty.
\end{equation}
Indeed, the upper estimate follows from
\[
  \mathbb{P}\Big[ \sup_{0\leq t\leq 1}|B_H(t)| \geq y\Big]\leq 
  2\, \mathbb{P}\Big[ \sup_{0\leq t\leq 1}B_H(t) \geq y\Big]
  \]
and the Borell-TIS inequality~\cite[Theorem~4.2]{No12}, and the lower one is clear
from $ \sup_{0\leq t\leq 1}|B_H(t)| \geq B_H(1)$.
The following result gives a large deviation estimate for  ggBm. For $\beta=1$, the distribution has a Gaussian upper tail,  of course.
For $0<\beta<1$, the decay is between exponential and Gaussian, which
is sometimes called compressed exponential.
\begin{theorem}\label{thm:ld}
   Let $0<\alpha<2$ and $0<\beta\leq 1$, and assume that $\|\cdot\|$
   is a norm on the H\"older space $C_0^\gamma[0,1]$, where $0<\gamma<H=\alpha/2$,
   such that
   \begin{equation}\label{eq fbm ld}
  \mathbb{P}\big[ \| B_H\| \geq y\big]
  =\exp\big({-\kappa}y^2+o(y^2)\big),\quad y\uparrow \infty,
\end{equation}
  for some $\kappa>0$.
   Then there are constants $K_1,K_2>0$ such that
   \begin{align}
      \exp\Big({-K_1}y^{\frac{2}{2-\beta}}\big(1+o(1)\big)\Big) &\leq
      \mathbb{P}\big[ \|B_{\alpha,\beta}\| \geq y\big]  \label{eq:ld low} \\
  &\leq \exp\Big({-K_2}y^{\frac{2}{2-\beta}}\big(1+o(1)\big)\Big),\quad y\uparrow \infty.
  \label{eq:ld}
   \end{align}
\end{theorem}
\begin{proof}
  We may assume $\beta<1$, because for $\beta=1$ we have $B_{\alpha,1}=B_H$
  and the assumption~\eqref{eq fbm ld}
  makes the statement trivial.
  With $\bar{F}_H=1-F_H$ the tail distribution function of $\|B_H\|$,
  we have, from~\eqref{eq:repr},
  \[
     \mathbb{P}\big[ \|B_{\alpha,\beta}\| \geq y\big]
     =\int_0^\infty \bar{F}_H(yx^{-1/2})M_\beta(x)dx.
  \]
  If $\kappa=\tfrac12$, then $\bar{F}_H$ satisfies
  \begin{equation}\label{eq:F b}
    \bar{F}_H(y)=\exp\big({-\tfrac12}y^2+o(y^2)\big),\quad y\uparrow \infty,
  \end{equation}
  by~\eqref{eq fbm ld}. We assume  $\kappa=\tfrac12$ for rest of the proof,
  as $\kappa>0$ is a trivial extension. Let $0<\hat \kappa <\tfrac12$ be arbitrary.
  Since $M_\beta$ is bounded, we obtain
  \begin{align*}
     \int_0^1 \bar{F}_H(yx^{-1/2})M_\beta(x)dx
     &\leq C \int_0^1 e^{-\hat{\kappa}y^2/x}M_\beta(x)dx\\
     &\leq C \int_0^1 e^{-\hat{\kappa}y^2/x}dx\\
     &=C\big(e^{-\hat{\kappa}y^2}-\hat{\kappa} y^2 \Gamma(0,\hat{\kappa} y^2)\big),
  \end{align*}
  where $\Gamma(a,z)=\int_z^\infty t^{a-1}e^{-t}dt$
  is the incomplete gamma function. Using a well-known expansion
  of that function \cite[\S8.11]{DLMF}, we conclude
  \begin{equation}\label{eq:int01}
      \int_0^1 \bar{F}_H(yx^{-1/2})M_\beta(x)dx
      \leq \exp\big({-\hat{\kappa}}y^2+o(y^2)\big),\quad y\uparrow \infty.
  \end{equation}
  As $\beta<1$, this is negligible compared to the decay rate
  claimed in~\eqref{eq:ld low} and~\eqref{eq:ld}.
   Now define $h(y):=y^2/(\log y)$.
    Since $\bar{F}_H \leq 1$, and using~\eqref{eq:M as}, we have
  \begin{align*}
    \int_{h(y)}^\infty  \bar{F}_H(yx^{-1/2})M_\beta(x)dx
    &\leq   \int_{h(y)}^\infty  M_\beta(x)dx\\
    &\leq \int_{h(y)}^\infty   \exp\big({-C}x^{\frac{1}{1-\beta}}\big)dx\\
    &\leq   \exp\big({-C}h(y)^{\frac{1}{1-\beta}}\big).
  \end{align*}
  Since $2/(1-\beta)>2/(2-\beta)$, this is of faster decay
   than~\eqref{eq:ld}.
  
  It remains to show that the integral $\int_1^{h(y)}\bar{F}_H(yx^{-1/2})M_\beta(x)dx$ has the claimed
  growth order~\eqref{eq:ld}.
  By dividing the exponent in~\eqref{eq:M as} by~2, which makes
  the decay slower, we obtain
   \begin{equation*}
     M_\beta(x) \leq C\exp\Big( {-\frac{1-\beta}{2\beta}}(\beta x)^{\frac{1}{1-\beta}}
     \Big),\quad x\geq1.
  \end{equation*}
  Similarly, \eqref{eq:F b} implies
  \[
    \bar{F}_H(y) \leq C e^{-y^2/3}, \quad y\geq 1.
  \]
  Analogously, we can \emph{increase} the constants in the exponents
  to find \emph{lower} estimates, for which the following reasoning is analogous,
  and yields~\eqref{eq:ld low}.
  Therefore, we only discuss the upper estimate for $\int_1^{h(y)}$.
  This is a straightforward application of the Laplace
  method~\cite[Chapter~4]{deB58} to
  the integral
  \[
     \int_1^{h(y)}\exp\Big({-\frac{y^2}{3x}} -\frac{1-\beta}{2\beta}(\beta x)^{\frac{1}{1-\beta}}
     \Big)dx,
  \]
  which results from the two preceding estimates.
  The integrand is a strictly concave function with a maximum at
  \[
    x_0(y)=c y^{\frac{2(1-\beta)}{2-\beta}} \in (1,h(y))
  \]
  for some constant $c>0$. As we are not concerned with lower order
  terms, it suffices to evaluate the integrand at $x_0(y)$ to conclude
  \[
    \int_1^{h(y)} \bar{F}(yx^{-1/2})M_\beta(x)dx
    \leq \exp\Big({-C}y^{\frac{2}{2-\beta}}(1+o(1))\Big).
  \]
  This completes the proof.
\end{proof}
We now comment on applying Theorem~\ref{thm:ld} to other norms than the sup norm, which
requires verifying~\eqref{eq fbm ld}. As mentioned above, for the sup norm, this
follows from the Borell-TIS inequality. For an arbitrary norm $\|\cdot\|$ on
H\"older space, we have
\[
  \mathbb{P}\big[ \| B_H\| \geq y\big] = \mathbb{P}\big[y^{-1}B_H\in
  \{\| f \| \geq 1\}\big].
\]
In principle, this is in the scope of the general LDP (large deviation principle)
for Gaussian measures \cite[Theorem 3.4.12]{DeSt89}, but it may not be trivial
to verify the assumptions. For $H=\tfrac12$ and the H\"older norm,
this was done in~\cite{BaBeKe92}, extending Schilder's theorem.
Note that choosing a stronger topology than the uniform one enlarges
the dual space of path space, making it harder to verify the defining property
of a Gaussian measure. For the H\"older topology, we are on safe grounds, though,
by another approach:
Using the double sum method for Gaussian fields, Fatalov
has shown that~\eqref{eq:fbm ld sup} holds for the $\gamma$-H\"older
norm~\cite[Theorem~1.3]{Fa03}, and so Theorem~\ref{thm:ld}
is applicable to this norm (with $0<\gamma<H$, of course).

\bibliographystyle{siam}
\bibliography{literature}

\begin{thebibliography}{10}

\bibitem{AuLiLi09}
{\sc F.~Aurzada, M.~Lifshits, and W.~Linde}, {\em Small deviations of stable
  processes and entropy of the associated random operators}, Bernoulli, 15
  (2009), pp.~1305--1334.

\bibitem{AuSi07}
{\sc F.~Aurzada and T.~Simon}, {\em Small ball probabilities for stable
  convolutions}, ESAIM Probab. Stat., 11 (2007), pp.~327--343.

\bibitem{BaBeKe92}
{\sc P.~Baldi, G.~Ben~Arous, and G.~Kerkyacharian}, {\em Large deviations and
  the {S}trassen theorem in {H}\"{o}lder norm}, Stochastic Process. Appl., 42
  (1992), pp.~171--180.

\bibitem{Br03}
{\sc J.~C. Bronski}, {\em Small ball constants and tight eigenvalue asymptotics
  for fractional {B}rownian motions}, J. Theoret. Probab., 16 (2003),
  pp.~87--100.

\bibitem{deB58}
{\sc N.~G. de~Bruijn}, {\em Asymptotic methods in analysis}, Bibliotheca
  Mathematica, Vol. IV, North-Holland Publishing Co., Amsterdam; P. Noordhoff
  Ltd., Groningen; Interscience Publishers Inc., New York, 1958.

\bibitem{DeSt89}
{\sc J.-D. Deuschel and D.~W. Stroock}, {\em Large deviations}, vol.~137 of
  Pure and Applied Mathematics, Academic Press, Inc., Boston, MA, 1989.

\bibitem{DLMF}
{\em {\it NIST Digital Library of Mathematical Functions}}.
\newblock http://dlmf.nist.gov/, Release 1.1.8 of 2022-12-15.
\newblock F.~W.~J. Olver, A.~B. {Olde Daalhuis}, D.~W. Lozier, B.~I. Schneider,
  R.~F. Boisvert, C.~W. Clark, B.~R. Miller, B.~V. Saunders, H.~S. Cohl, and
  M.~A. McClain, eds.

\bibitem{El05}
{\sc J.~Elstrodt}, {\em Ma\ss - und {I}ntegrationstheorie}, Springer-Verlag,
  Berlin, fourth~ed., 2005.

\bibitem{Fa03}
{\sc V.~R. Fatalov}, {\em Large deviations for {G}aussian processes in the
  {H}\"{o}lder norm}, Izv. Ross. Akad. Nauk Ser. Mat., 67 (2003), pp.~207--224.

\bibitem{GoKiMaRo14}
{\sc R.~Gorenflo, A.~A. Kilbas, F.~Mainardi, and S.~V. Rogosin}, {\em
  Mittag-{L}effler functions, related topics and applications}, Springer
  Monographs in Mathematics, Springer, Heidelberg, 2014.

\bibitem{Ko16}
{\sc K.~Kobayashi}, {\em Small ball probabilities for a class of time-changed
  self-similar processes}, Statist. Probab. Lett., 110 (2016), pp.~155--161.

\bibitem{LaNu03}
{\sc N.~Lanjri~Zadi and D.~Nualart}, {\em Smoothness of the law of the supremum
  of the fractional {B}rownian motion}, Electron. Comm. Probab., 8 (2003),
  pp.~102--111.

\bibitem{LiSh01}
{\sc W.~V. Li and Q.-M. Shao}, {\em Gaussian processes: inequalities, small
  ball probabilities and applications}, in Stochastic processes: theory and
  methods, vol.~19 of Handbook of Statist., North-Holland, Amsterdam, 2001,
  pp.~533--597.

\bibitem{Li12}
{\sc M.~Lifshits}, {\em Lectures on {G}aussian processes}, Springer Briefs in
  Mathematics, Springer, Heidelberg, 2012.

\bibitem{MaMuPa10}
{\sc F.~Mainardi, A.~Mura, and G.~Pagnini}, {\em The {$M$}-{W}right function in
  time-fractional diffusion processes: a tutorial survey}, Int. J. Differ.
  Equ.,  (2010), pp.~1--29, Art. ID 104505.

\bibitem{MuMa09}
{\sc A.~Mura and F.~Mainardi}, {\em A class of self-similar stochastic
  processes with stationary increments to model anomalous diffusion in
  physics}, Integral Transforms Spec. Funct., 20 (2009), pp.~185--198.

\bibitem{MuPa08}
{\sc A.~Mura and G.~Pagnini}, {\em Characterizations and simulations of a class
  of stochastic processes to model anomalous diffusion}, J. Phys. A, 41 (2008),
  pp.~285003, 22.

\bibitem{Na09}
{\sc E.~Nane}, {\em Laws of the iterated logarithm for a class of iterated
  processes}, Statist. Probab. Lett., 79 (2009), pp.~1744--1751.

\bibitem{No12}
{\sc I.~Nourdin}, {\em Selected aspects of fractional {B}rownian motion},
  vol.~4 of Bocconi \& Springer Series, Springer, Milan; Bocconi University
  Press, Milan, 2012.

\bibitem{NoVa99}
{\sc A.~Novikov and E.~Valkeila}, {\em On some maximal inequalities for
  fractional {B}rownian motions}, Statist. Probab. Lett., 44 (1999),
  pp.~47--54.

\bibitem{Pr14}
{\sc B.~L.~S. Prakasa~Rao}, {\em Maximal inequalities for fractional {B}rownian
  motion: an overview}, Stoch. Anal. Appl., 32 (2014), pp.~450--479.

\bibitem{Sc90}
{\sc W.~R. Schneider}, {\em Grey noise}, in Stochastic processes, physics and
  geometry ({A}scona and {L}ocarno, 1988), World Sci. Publ., Teaneck, NJ, 1990,
  pp.~676--681.

\bibitem{Wr40}
{\sc E.~M. Wright}, {\em The generalized {B}essel function of order greater
  than one}, Quart. J. Math. Oxford Ser., 11 (1940), pp.~36--48.

\end{thebibliography}

\end{document}